\documentclass[12pt]{amsart}
\usepackage{amsmath,amsthm,amsfonts,amssymb,mathrsfs}
\usepackage{color}

\usepackage{tikz}

\usepackage{amssymb,cite}
\usepackage[colorlinks,plainpages,citecolor=magenta, linkcolor=blue, backref]{hyperref}

\usepackage{hyperref}
\date{\today}

\usepackage{hyperref}
\input{xy}
 \xyoption{all}
 \xyoption{arc}

\newtheorem{theorem}{Theorem}

\newtheorem{proposition}{Proposition}
\newtheorem{corollary}{Corollary}

\newtheorem{lemma}{Lemma}
\theoremstyle{definition}

\newtheorem{remark}{Remark}

\begin{document}

\title[On injective  endomorphisms of the semigroup $\boldsymbol{B}_{\omega}^{\mathscr{F}^3}$]{On injective endomorphisms of the semigroup $\boldsymbol{B}_{\omega}^{\mathscr{F}^3}$ with a three-element family $\mathscr{F}^3$ of inductive non-empty subsets of $\omega$}
\author{Oleg Gutik and Marko Serivka}
\address{Ivan Franko National University of Lviv, Universytetska 1, Lviv, 79000, Ukraine}
\email{oleg.gutik@lnu.edu.ua, ogutik@gmail.com, marko.serivka@lnu.edu.ua}



\begin{abstract}
We describe injective endomorphisms of the semigroup $\boldsymbol{B}_{\omega}^{\mathscr{F}^3}$ with a three-element family $\mathscr{F}^3$ of inductive non-empty subsets of $\omega$.  In particular we find endomorphisms $\varpi_3$ and $\lambda$ of $\boldsymbol{B}_{\omega}^{\mathscr{F}^3}$ such that for every injective endomorphism $\varepsilon$ of the semigroup $\boldsymbol{B}_{\omega}^{\mathscr{F}^3}$ there exists an injective endomorphism $\iota\in\left\langle\lambda,\varpi_3\right\rangle$ such that $\varepsilon=\alpha_{[k]}\circ\iota$ for some positive integer $k$, where $\alpha_{[k]}$ is an injective monoid endomorphism of $\boldsymbol{B}_{\omega}^{\mathscr{F}^3}$.
\end{abstract}

\maketitle


\section{\textbf{Introduction, motivation and main definitions}}

We shall follow the terminology of the monographs~\cite{Clifford-Preston-1961, Clifford-Preston-1967, Lawson=1998}. By $\omega$ we denote the set of all non-negative integers and by $\mathbb{N}$ the set of all positive integers. If $\mathfrak{f}\colon X\to Y$ is a map, then by $(x)\mathfrak{f}$ and $(A)\mathfrak{f}$ we denote the image of $x\in X$ and $A\subseteq X$ under $\mathfrak{f}$, respectively.

\smallskip

Let $\mathscr{P}(\omega)$ be  the family of all subsets of $\omega$. For any $F\in\mathscr{P}(\omega)$ and any integer $n$ we put $n+F=\{n+k\colon k\in F\}$ if $F\neq\varnothing$ and $n+\varnothing=\varnothing$.
A subfamily $\mathscr{F}\subseteq\mathscr{P}(\omega)$ is called \emph{${\omega}$-closed} if $F_1\cap(-n+F_2)\in\mathscr{F}$ for all $n\in\omega$ and $F_1,F_2\in\mathscr{F}$. For any $a\in\omega$ we denote $[a)=\{x\in\omega\colon x\geqslant a\}$.

\smallskip

A subset $A$ of $\omega$ is said to be \emph{inductive}, if $i\in A$ implies $i+1\in A$. Obvious, that $\varnothing$ is an inductive subset of $\omega$.

\begin{remark}[\!\!{\cite{Gutik-Mykhalenych=2021}}]\label{remark-1.1}
\begin{enumerate}
  \item\label{remark-1.1(1)} By Lemma~6 from \cite{Gutik-Mykhalenych=2020} non-empty subset $F\subseteq \omega$ is inductive in $\omega$ if and only if $(-1+F)\cap F=F$.
  \item\label{remark-1.1(2)} Since the set $\omega$ with the usual order is well-ordered, for any non-empty inductive subset $F$ in $\omega$ there exists non-negative integer $n_F\in\omega$ such that $[n_F)=F$.
  \item\label{remark-1.1(3)} Statement \eqref{remark-1.1(2)} implies that the intersection of an arbitrary finite family of non-empty inductive subsets in $\omega$ is a non-empty inductive subset of  $\omega$.
\end{enumerate}
\end{remark}

For an arbitrary semigroup $S$ any homomorphism $\alpha\colon S\to S$ is called an \emph{endomorphism} of $S$. If the semigroup has the identity element $1_S$ then the endomorphism $\alpha$ of $S$ such that $(1_S)\alpha=1_S$ is said to be a \emph{monoid endomorphism} of $S$. A bijective endomorphism of $S$ is called an \emph{automorphism}.

\smallskip

A semigroup $S$ is called {\it inverse} if for any
element $x\in S$ there exists a unique $x^{-1}\in S$ such that
$xx^{-1}x=x$ and $x^{-1}xx^{-1}=x^{-1}$. The element $x^{-1}$ is
called the {\it inverse of} $x\in S$. If $S$ is an inverse
semigroup, then the function $\operatorname{inv}\colon S\to S$
which assigns to every element $x$ of $S$ its inverse element
$x^{-1}$ is called the {\it inversion}.

\smallskip


If $S$ is a semigroup, then we shall denote the subset of all
idempotents in $S$ by $E(S)$. If $S$ is an inverse semigroup, then
$E(S)$ is closed under multiplication and we shall refer to $E(S)$ as a
\emph{band} (or the \emph{band of} $S$). Then the semigroup
operation on $S$ determines the following partial order $\preccurlyeq$
on $E(S)$:
\begin{center}
$e\preccurlyeq f$ if and only if $ef=fe=e$.
\end{center}
This order is
called the {\em natural partial order} on $E(S)$. 

\smallskip

If $S$ is an inverse semigroup then the semigroup operation on $S$ determines the following partial order $\preccurlyeq$
on $S$: $s\preccurlyeq t$ if and only if there exists $e\in E(S)$ such that $s=te$. This order is
called the {\em natural partial order} on $S$ \cite{Wagner-1952}.

\smallskip

The \emph{bicyclic monoid} ${\mathscr{C}}(p,q)$ is the semigroup with the identity $1$ generated by two elements $p$ and $q$ subjected only to the condition $pq=1$. The semigroup operation on ${\mathscr{C}}(p,q)$ is determined as
follows:
\begin{equation*}
    q^kp^l\cdot q^mp^n=q^{k+m-\min\{l,m\}}p^{l+n-\min\{l,m\}}.
\end{equation*}
It is well known that the bicyclic monoid ${\mathscr{C}}(p,q)$ is a bisimple (and hence simple) combinatorial $E$-unitary inverse semigroup and every non-trivial congruence on ${\mathscr{C}}(p,q)$ is a group congruence \cite{Clifford-Preston-1961}.

\smallskip

On the set $\boldsymbol{B}_{\omega}=\omega\times\omega$ we define the semigroup operation ``$\cdot$'' in the following way
\begin{equation}\label{eq-1.1}
  (i_1,j_1)\cdot(i_2,j_2)=
  \left\{
    \begin{array}{ll}
      (i_1-j_1+i_2,j_2), & \hbox{if~} j_1\leqslant i_2;\\
      (i_1,j_1-i_2+j_2), & \hbox{if~} j_1\geqslant i_2.
    \end{array}
  \right.
\end{equation}
It is well known that the bicyclic monoid $\mathscr{C}(p,q)$ is isomorphic to the semigroup $\boldsymbol{B}_{\omega}$ by the mapping $\mathfrak{h}\colon \mathscr{C}(p,q)\to \boldsymbol{B}_{\omega}$, $q^kp^l\mapsto (k,l)$ (see: \cite[Section~1.12]{Clifford-Preston-1961} or \cite[Exercise IV.1.11$(ii)$]{Petrich-1984}).

\smallskip

Next we shall describe the construction which is introduced in \cite{Gutik-Mykhalenych=2020}.

\smallskip

Let $\mathscr{F}$ be an ${\omega}$-closed subfamily of $\mathscr{P}(\omega)$. On the set $\boldsymbol{B}_{\omega}\times\mathscr{F}$ we define the semigroup operation ``$\cdot$'' in the following way
\begin{equation}\label{eq-1.2}
  (i_1,j_1,F_1)\cdot(i_2,j_2,F_2)=
  \left\{
    \begin{array}{ll}
      (i_1-j_1+i_2,j_2,(j_1-i_2+F_1)\cap F_2), & \hbox{if~} j_1\leqslant i_2;\\
      (i_1,j_1-i_2+j_2,F_1\cap (i_2-j_1+F_2)), & \hbox{if~} j_1\geqslant i_2.
    \end{array}
  \right.
\end{equation}
In \cite{Gutik-Mykhalenych=2020} is proved that if the family $\mathscr{F}\subseteq\mathscr{P}(\omega)$ is ${\omega}$-closed then $(\boldsymbol{B}_{\omega}\times\mathscr{F},\cdot)$ is a semigroup. Moreover, if an ${\omega}$-closed family  $\mathscr{F}\subseteq\mathscr{P}(\omega)$ contains the empty set $\varnothing$ then the set
$ 
  \boldsymbol{I}=\{(i,j,\varnothing)\colon i,j\in\omega\}
$ 
is an ideal of the semigroup $(\boldsymbol{B}_{\omega}\times\mathscr{F},\cdot)$. For any ${\omega}$-closed family $\mathscr{F}\subseteq\mathscr{P}(\omega)$ the following semigroup
\begin{equation*}
  \boldsymbol{B}_{\omega}^{\mathscr{F}}=
\left\{
  \begin{array}{ll}
    (\boldsymbol{B}_{\omega}\times\mathscr{F},\cdot)/\boldsymbol{I}, & \hbox{if~} \varnothing\in\mathscr{F};\\
    (\boldsymbol{B}_{\omega}\times\mathscr{F},\cdot), & \hbox{if~} \varnothing\notin\mathscr{F}
  \end{array}
\right.
\end{equation*}
is defined in \cite{Gutik-Mykhalenych=2020}. The semigroup $\boldsymbol{B}_{\omega}^{\mathscr{F}}$ generalizes the bicyclic monoid and the countable semigroup of matrix units. In \cite{Gutik-Mykhalenych=2020} it is proven that $\boldsymbol{B}_{\omega}^{\mathscr{F}}$ is a combinatorial inverse semigroup and Green's relations, the natural partial order on $\boldsymbol{B}_{\omega}^{\mathscr{F}}$ and its set of idempotents are described.
Also, in \cite{Gutik-Mykhalenych=2020} the criteria when the semigroup $\boldsymbol{B}_{\omega}^{\mathscr{F}}$ is simple, $0$-simple, bisimple, $0$-bisimple, or it has the identity, are given.
In particularly in \cite{Gutik-Mykhalenych=2020} it is proven that the semigroup $\boldsymbol{B}_{\omega}^{\mathscr{F}}$ is isomorphic to the semigrpoup of ${\omega}{\times}{\omega}$-matrix units if and only if $\mathscr{F}$ consists of a singleton set and the empty set, and $\boldsymbol{B}_{\omega}^{\mathscr{F}}$ is isomorphic to the bicyclic monoid if and only if $\mathscr{F}$ consists of a non-empty inductive subset of $\omega$.

\smallskip

Group congruences on the semigroup  $\boldsymbol{B}_{\omega}^{\mathscr{F}}$ and its homomorphic retracts  in the case when an ${\omega}$-closed family $\mathscr{F}$ consists of inductive non-empty subsets of $\omega$ are studied in \cite{Gutik-Mykhalenych=2021}. It is proven that a congruence $\mathfrak{C}$ on $\boldsymbol{B}_{\omega}^{\mathscr{F}}$ is a group congruence if and only if its restriction on a subsemigroup of $\boldsymbol{B}_{\omega}^{\mathscr{F}}$, which is isomorphic to the bicyclic semigroup, is not the identity relation. Also in \cite{Gutik-Mykhalenych=2021}, all non-trivial homomorphic retracts and isomorphisms  of the semigroup $\boldsymbol{B}_{\omega}^{\mathscr{F}}$ are described. In \cite{Gutik-Mykhalenych=2022} it is proven that an injective endomorphism $\varepsilon$ of the semigroup $\boldsymbol{B}_{\omega}^{\mathscr{F}}$ is the indentity transformation if and only if  $\varepsilon$ has three distinct fixed points, which is equivalent to existence non-idempotent element $(i,j,[p))\in\boldsymbol{B}_{\omega}^{\mathscr{F}}$ such that  $(i,j,[p))\varepsilon=(i,j,[p))$.

\smallskip

In \cite{Gutik-Lysetska=2021, Lysetska=2020} the algebraic structure of the semigroup $\boldsymbol{B}_{\omega}^{\mathscr{F}}$ is established in the case when ${\omega}$-closed family $\mathscr{F}$ consists of atomic subsets of ${\omega}$. The structure of the semigroup $\boldsymbol{B}_{\omega}^{\mathscr{F}_n}$, for the family  $\mathscr{F}_n$ which is generated by the initial interval $\{0,1,\ldots,n\}$ of $\omega$, is studied in \cite{Gutik-Popadiuk=2023}. The semigroup of endomorphisms of $\boldsymbol{B}_{\omega}^{\mathscr{F}_n}$ is described in \cite{Gutik-Popadiuk=2022, Popadiuk=2022}.

\smallskip

In \cite{Gutik-Prokhorenkova-Sekh=2021} it is proven that the semigroup $\mathrm{\mathbf{End}}(\boldsymbol{B}_{\omega})$ of all endomorphisms of the bicyclic semigroup $\boldsymbol{B}_{\omega}$ is isomorphic to the semidirect products $(\omega,+)\rtimes_\varphi(\omega,*)$, where $+$ and $*$ are the usual addition and the usual multiplication on the set of non-negative integers $\omega$, respectively.

\smallskip

In the paper \cite{Gutik-Pozdniakova=2022}  injective endomorphisms of the semigroup $\boldsymbol{B}_{\omega}^{\mathscr{F}}$ with the two-elements family $\mathscr{F}$ of inductive non-empty subsets of $\omega$ are studies. Also, in \cite{Gutik-Pozdniakova=2022} the authors describe the elements of the semigroup $\boldsymbol{End}^1_*(\boldsymbol{B}_{\omega}^{\mathscr{F}})$ of all injective monoid endomorphisms of the monoid $\boldsymbol{B}_{\omega}^{\mathscr{F}}$, and show that Green's relations $\mathscr{R}$, $\mathscr{L}$, $\mathscr{H}$, $\mathscr{D}$, and $\mathscr{J}$  on $\boldsymbol{End}^1_*(\boldsymbol{B}_{\omega}^{\mathscr{F}})$ coincide with the relation of equality. In \cite{Gutik-Pozdniakova=2023a, Gutik-Pozdniakova=2023b} the semigroup $\boldsymbol{End}^1(\boldsymbol{B}_{\omega}^{\mathscr{F}})$ of all  monoid endomorphisms of the monoid $\boldsymbol{B}_{\omega}^{\mathscr{F}}$ is studied.

In \cite{Gutik-Serivka=2025} we study the semigroup $\overline{\boldsymbol{End}}(\boldsymbol{B}_{\omega}^{\mathcal{F}^2})$ of all endomorphisms of the bicyclic extension $\boldsymbol{B}_{\omega}^{\mathcal{F}^2}$ with the two-element family $\mathcal{F}^2$ of inductive non-empty subsets of $\omega$. The submonoid  $\left\langle\varpi\right\rangle^1$ of $\overline{\boldsymbol{End}}(\boldsymbol{B}_{\omega}^{\mathcal{F}^2})$ with the property that every element of the semigroup $\overline{\boldsymbol{End}}(\boldsymbol{B}_{\omega}^{\mathcal{F}^2})$ has the unique representation as the product of the monoid endomorphism of $\boldsymbol{B}_{\omega}^{\mathcal{F}^2}$ and the element of $\left\langle\varpi\right\rangle^1$ is constructed.

\smallskip

Later we assume that $\mathscr{F}^3$ is a family of inductive non-empty subsets of $\omega$ which consists of three sets.  By Proposition~1 of \cite{Gutik-Mykhalenych=2021} for any $\omega$-closed family $\mathscr{F}$ of inductive subsets in $\mathscr{P}(\omega)$ there exists an $\omega$-closed family $\mathscr{F}^*$ of inductive subsets in $\mathscr{P}(\omega)$ such that $[0)\in \mathscr{F}^*$ and the semigroups $\boldsymbol{B}_{\omega}^{\mathscr{F}}$ and $\boldsymbol{B}_{\omega}^{\mathscr{F}^*}$ are isomorphic. Hence without loss of generality we may assume that the family $\mathscr{F}$ contains the set $[0)$, i.e., $\mathscr{F}^3=\left\{[0),[1),[2)\right\}$.

\smallskip

In \cite{Gutik-Serivka=2023} we describe injective monoid endomorphisms of the semigroup $\boldsymbol{B}_{\omega}^{\mathscr{F}^3}$. Here we shown that for
every injective monoid endomorphism $\varepsilon$ of $\boldsymbol{B}_{\omega}^{\mathscr{F}^3}$ there exists a positive integer $k$ such that $\varepsilon=\alpha_{[k]}$ and the mapping $\alpha_{[k]}\colon \boldsymbol{B}_{\omega}^{\mathscr{F}^3}\to\boldsymbol{B}_{\omega}^{\mathscr{F}^3}$ is defined by the formula
\begin{equation*}
  (i,j,[p))\alpha_{[k]}=
  \left\{
    \begin{array}{ll}
      (ki,kj,[p)), & \hbox{if~} p\in\{0,1\};\\
      (k(i+1)-1,k(j+1)-1,[2)), & \hbox{if~} p=2,
    \end{array}
  \right.
\end{equation*}
for all $i,j\in\omega$ (see Theorem~5 of \cite{Gutik-Serivka=2023}). Also we prove that the monoid $\boldsymbol{End}_{\textsf{inj}}^1(\boldsymbol{B}_{\omega}^{\mathscr{F}^3})$ of all injective monoid endomorphisms of the semigroup $\boldsymbol{B}_{\omega}^{\mathscr{F}^3}$ is isomorphic to the multiplicative semigroup of positive integers.

In this paper we describe injective endomorphisms of the semigroup $\boldsymbol{B}_{\omega}^{\mathscr{F}^3}$. In particular we find endomorphisms $\varpi_3$ and $\lambda$ of $\boldsymbol{B}_{\omega}^{\mathscr{F}^3}$ such that for every injective endomorphism $\varepsilon$ of the semigroup $\boldsymbol{B}_{\omega}^{\mathscr{F}^3}$ there exists an injective endomorphism $\iota\in\left\langle\lambda,\varpi_3\right\rangle$ such that $\varepsilon=\alpha_{[k]}\circ\iota$ for some positive integer $k$.

\section{\textbf{On some properties of injective endomorphisms of the semigroup $\boldsymbol{B}_{\omega}^{\mathscr{F}^n}$}}\label{section-2}

For any positive integer $n$ we denote $\mathscr{F}^n=\{[0),\ldots[n-1)\}$.

\begin{proposition}\label{proposition-2.1}
For any positive integer $n$ the following statements hold:
\begin{itemize}
  \item[$(i)$] the map $\lambda\colon \boldsymbol{B}_{\omega}^{\mathscr{F}^n}\to \boldsymbol{B}_{\omega}^{\mathscr{F}^n}$ defined by the formula
  \begin{equation*}
  (i,j,[p))\lambda=(i+1,j+1,[p)), \quad  i,j\in\omega, \; p\in\{0,\ldots,n-1\},
  \end{equation*}
  is an injective endomorphism of the semigroup $\boldsymbol{B}_{\omega}^{\mathscr{F}^n}$;

  \item[$(ii)$] the map $\varpi_n\colon \boldsymbol{B}_{\omega}^{\mathscr{F}^n}\to \boldsymbol{B}_{\omega}^{\mathscr{F}^n}$ defined by the formula
  \begin{equation*}
  (i,j,[p))\varpi_n=(i+p,j+p,[n-1-p)), \quad  i,j\in\omega, \; p\in\{0,\ldots,n-1\},
  \end{equation*}
  is an injective endomorphism of the semigroup $\boldsymbol{B}_{\omega}^{\mathscr{F}^n}$;

  \item[$(iii)$] $\varpi_n^2=\lambda^{n-1}$;

  \item[$(iv)$] $\varpi_n\lambda=\lambda\varpi_n$.
\end{itemize}
\end{proposition}

\begin{proof}
The proofs of statements $(i)$ and $(ii)$ are same as Proposition~6 and Lemma~3 of \cite{Gutik-Pozdniakova=2023c}, respectively.

$(iii)$ For arbitrary $i,j\in\omega$ and $p\in\{0,\ldots,n-1\}$ we have that
\begin{align*}
  (i,j,[p))\varpi_n^2&=(i+p,j+p,[n-1-p))\varpi_n= \\
   &=(i+p+n-1-p,j+p+n-1-p,[n-1-(n-1-p)))= \\
   &=(i+n-1,j+n-1,[p)),
\end{align*}
and hence $\varpi_n^2=\lambda^{n-1}$.

$(iv)$ For arbitrary $i,j\in\omega$ and $p\in\{0,\ldots,n-1\}$ we get that
\begin{align*}
  (i,j,[p))\varpi_n\lambda&=(i+p,j+p,[n-1-p))\lambda= \\
   &=(i+p+1,j+p+1,[n-1-p))
\end{align*}
and
\begin{align*}
  (i,j,[p))\lambda\varpi_n&=(i+1,j+1,[p))\varpi_n=\\
   &=(i+1+p,j+1+p,[n-1-p)),
\end{align*}
which implies the equality $\varpi_n\lambda=\lambda\varpi_n$.
\end{proof}

In \cite{Gutik-Mykhalenych=2020} it was proven that the semigroup  $\boldsymbol{B}_{\omega}^{\mathscr{F}}$ is isomorphic to the bicyclic semigroup if and only if $\mathscr{F}=[p)$ for some $p\in\omega$.

For any non-negative integers $s<t\leqslant n$ we denote $\mathscr{F}^{[s,t]}=\{[s),\ldots,[t)\}$ and
\begin{equation*}
  \boldsymbol{B}_{\omega}^{\mathscr{F}^{[s,t]}}=\{(i,j,[p))\colon i,j\in\omega, \; p=s,\ldots,t\}.
\end{equation*}

\begin{proposition}\label{proposition-2.2}
Let $s$ and $t$ be non-negative integers such that $s<t\leqslant n$. Then the subsemigroup $\boldsymbol{B}_{\omega}^{\mathscr{F}^{[s,t]}}$ of $\boldsymbol{B}_{\omega}^{\mathscr{F}^{n}}$ is isomorphic to the semigroup $\boldsymbol{B}_{\omega}^{\mathscr{F}^{m}}$, $m\leqslant n$, if and only if $m=t-s+1$.
\end{proposition}

\begin{proof}
$(\Leftarrow)$ Suppose that $m=t-s+1$. We define the map $\mathfrak{I}\colon \boldsymbol{B}_{\omega}^{\mathscr{F}^{[s,t]}}\to \boldsymbol{B}_{\omega}^{\mathscr{F}^{m}}$ by the formula
\begin{equation*}
  (i,j,[p+s))\mathfrak{I}=(i,j,[p)), \qquad i,j\in\omega, \; p\in\{0,\ldots,m-1\}.
\end{equation*}
For arbitrary $i_1,i_2,j_1,j_2\in\omega$ and $p_1,p_2\in\{0,\ldots,m-1\}$ we have that
\begin{align*}
  ((i_1,j_1,[p_1+s))\cdot(i_2,j_2,[p_2+s)))\mathfrak{I}&=
  \left\{
    \begin{array}{ll}
      (i_1-j_1+i_2,j_2,(j_1-i_2+[p_1+s))\cap[p_2+s))\mathfrak{I}, & \hbox{if~} j_1<i_2;\\
      (i_1,j_2,[p_1+s)\cap[p_2+s))\mathfrak{I},                   & \hbox{if~} j_1=i_2;\\
      (i_1,j_1-i_2+j_2,[p_1+s)\cap(i_2-j_2+[p_2+s)))\mathfrak{I}, & \hbox{if~} j_1>i_2;
    \end{array}
  \right.=
  \\
   &=\left\{
    \begin{array}{ll}
      (i_1-j_1+i_2,j_2,(j_1-i_2+[p_1))\cap[p_2))\mathfrak{I}, & \hbox{if~} j_1<i_2;\\
      (i_1,j_2,[p_1)\cap[p_2))\mathfrak{I},                   & \hbox{if~} j_1=i_2;\\
      (i_1,j_1-i_2+j_2,[p_1)\cap(i_2-j_2+[p_2)))\mathfrak{I}, & \hbox{if~} j_1>i_2;
    \end{array}
  \right.=
  \\
  &=(i_1,j_1,[p_1))\cdot(i_2,j_2,[p_2))=\\
  &=(i_1,j_1,[p_1+s))\mathfrak{I}\cdot(i_2,j_2,[p_2+s))\mathfrak{I},
\end{align*}
because $[l+s)\cap[q+s)=([l)\cap[q))+s$ for all $l,q,s\in\omega$.

By Theorem~2$(iv)$ of \cite{Gutik-Mykhalenych=2020} the semigroup $\boldsymbol{B}_{\omega}^{\mathscr{F}^{[s,t]}}$ has $t-s+1$ $\mathscr{D}$-classes and $\boldsymbol{B}_{\omega}^{\mathscr{F}^{m}}$ has $m$ $\mathscr{D}$-classes.
If the semigroups $\boldsymbol{B}_{\omega}^{\mathscr{F}^{[s,t]}}$ and $\boldsymbol{B}_{\omega}^{\mathscr{F}^{m}}$ are isomorphic, then $\boldsymbol{B}_{\omega}^{\mathscr{F}^{[s,t]}}$ and $\boldsymbol{B}_{\omega}^{\mathscr{F}^{m}}$ have the same many $\mathscr{D}$-classes, and hence $m=t-s+1$.
\end{proof}

For any positive integer $m$ and $p\in\{1,\ldots,n-1\}$ we define
\begin{equation*}
  \boldsymbol{B}_{\omega}^{\mathscr{F}^{n}}[(m,m,[p))]=(m,m,[p))\cdot\boldsymbol{B}_{\omega}^{\mathscr{F}^{n}}\cdot(m,m,[p))= (m,m,[p))\cdot\boldsymbol{B}_{\omega}^{\mathscr{F}^{n}}\cap\boldsymbol{B}_{\omega}^{\mathscr{F}^{n}}\cdot(m,m,[p)).
\end{equation*}

\begin{lemma}\label{lemma-2.3}
Let $m$ be an arbitrary positive integer. Then $\boldsymbol{B}_{\omega}^{\mathscr{F}^{n}}[(m,m,[0))]=(\boldsymbol{B}_{\omega}^{\mathscr{F}^{n}})\lambda^m$.
\end{lemma}

\begin{proof}
For an arbitrary positive integer $m$, any $i,j\in\omega$ and $p\in\{0,\ldots,n-1\}$ we have that
\begin{align*}
  (m,m,[0))&\cdot(i,j,[p)))\cdot(m,m,[0))=
    \left\{
      \begin{array}{ll}
        (i,j,(m-i+[0))\cap[p))\cdot(m,m,[0)),     & \hbox{if~} m<i;\\
        (m,j,[0)\cap[p))\cdot(m,m,[0)),           & \hbox{if~} m=i;\\
        (m,m-i+j,[0)\cap(i-m+[p)))\cdot(m,m,[0)), & \hbox{if~} m>i
      \end{array}
    \right.=
    \\
   &=\left\{
      \begin{array}{ll}
        (i,j,[p))\cdot(m,m,[0)),           & \hbox{if~} m<i;\\
        (m,j,[p))\cdot(m,m,[0)),           & \hbox{if~} m=i;\\
        (m,m-i+j,[0))\cdot(m,m,[0)),       & \hbox{if~} m>i \hbox{~and~} i-m+p\leqslant 0;\\
        (m,m-i+j,(i-m+[p)))\cdot(m,m,[0)), & \hbox{if~} m>i \hbox{~and~} i-m+p>0
      \end{array}
    \right.=
    \\
   &=\left\{
      \begin{array}{ll}
        (i,j,[p))\cdot(m,m,[0)),           & \hbox{if~} m\leqslant i;\\
        (m,m-i+j,[0))\cdot(m,m,[0)),       & \hbox{if~} m>i \hbox{~and~} i-m+p\leqslant 0;\\
        (m,m-i+j,(i-m+[p)))\cdot(m,m,[0)), & \hbox{if~} m>i \hbox{~and~} i-m+p>0
      \end{array}
    \right.=
\end{align*}

\begin{align*}
    &=\left\{
      \begin{array}{ll}
        (i-j+m,m,(j-m+[p))\cap[0)),           & \hbox{if~} m\leqslant i \hbox{~and~} j<m;\\
        (i,j,[p)\cap[0)),                     & \hbox{if~} m\leqslant i \hbox{~and~} j=m;\\
        (i,j-m+m,[p)\cap(m-j+[0))),           & \hbox{if~} m\leqslant i \hbox{~and~} j>m;\\
        (m-m+i{-}j+m,m,(j{-}i+[0)\cap[0)),    & \hbox{if~} m>i, i+p\leqslant m \hbox{~and~} m-i+j<m;\\
        (m,m,[0)\cap[0)),                     & \hbox{if~} m>i, i+p\leqslant m \hbox{~and~} m-i+j=m;\\
        (m,m-i+j,[0)\cap(i-j+[0))),           & \hbox{if~} m>i, i+p\leqslant m \hbox{~and~} m-i+j>m;\\
        (m{-}m+i{-}j+m,m,(j{-}m+[p))\cap[0)), & \hbox{if~} m>i, i+p>m \hbox{~and~} m-i+j<m;\\
        (m,m,(i-m+[p))\cap[0)),               & \hbox{if~} m>i, i+p>m \hbox{~and~} m-i+j=m;\\
        (m,m{-}i+j,(i-m+[p))\cap(i{-}j+[0))), & \hbox{if~} m>i, i+p>m \hbox{~and~} m-i+j>m
      \end{array}
    \right.=
\\
    &=\left\{
      \begin{array}{ll}
        (i-j+m,m,(j-m+[p))\cap[0)),           & \hbox{if~} m\leqslant i \hbox{~and~} j<m;\\
        (i,j,[p)),               & \hbox{if~} m\leqslant i \hbox{~and~} j=m;\\
        (i,j,[p)),           & \hbox{if~} m\leqslant i \hbox{~and~} j>m;\\
        (m+i-j,m,[0)),       & \hbox{if~} m>i, i+p\leqslant m \hbox{~and~} j<i;\\
        (m,m,[0)),       & \hbox{if~} m>i, i+p\leqslant m \hbox{~and~} j=i;\\
        (m,m-i+j,[0)),       & \hbox{if~} m>i, i+p\leqslant m \hbox{~and~} j>i;\\
        (m+i-j,m,j-m+[p)), & \hbox{if~} m>i, i+p>m \hbox{~and~} j<i;\\
        (m,m,i-m+[p)), & \hbox{if~} m>i, i+p>m \hbox{~and~} j=i;\\
        (m,m{-}i+j,(i-m+[p))), & \hbox{if~} m>i, i+p>m \hbox{~and~} j>i,
      \end{array}
    \right.
\end{align*}
and hence we get that $\boldsymbol{B}_{\omega}^{\mathscr{F}^{n}}[(m,m,[0))]\subseteq (\boldsymbol{B}_{\omega}^{\mathscr{F}^{n}})\lambda^m$. Since for any $i,j\geqslant m$, $i,j\in\omega$ and $p\in\{0,\ldots,n\}$ we have that
\begin{equation*}
  (m,m,[0))\cdot(i,j,[p)))\cdot(m,m,[0))=(i,j,[p)))\cdot(m,m,[0))=(i,j,[p))),
\end{equation*}
we conclude that $(\boldsymbol{B}_{\omega}^{\mathscr{F}^{n}})\lambda^m \subseteq \boldsymbol{B}_{\omega}^{\mathscr{F}^{n}}[(m,m,[0))]$. This implies the requested equality.
\end{proof}

Proposition~\ref{proposition-2.1}$(i)$ and Lemma~\ref{lemma-2.3} imply the following corollary.

\begin{corollary}\label{corollary-2.4}
For any positive integer $m$ the semigroups $\boldsymbol{B}_{\omega}^{\mathscr{F}^{n}}$ and $\boldsymbol{B}_{\omega}^{\mathscr{F}^{n}}[(m,m,[0))]$ are isomorphic.
\end{corollary}

We observe that the isomorphism of the semigroups $\boldsymbol{B}_{\omega}^{\mathscr{F}^{n}}$ and $\boldsymbol{B}_{\omega}^{\mathscr{F}^{n}}[(m,m,[0))]$ is defined as the corestriction of the endomorphism $\lambda^m\colon \boldsymbol{B}_{\omega}^{\mathscr{F}^n}\to \boldsymbol{B}_{\omega}^{\mathscr{F}^n}$ onto its image $(\boldsymbol{B}_{\omega}^{\mathscr{F}^n})\lambda^m$.

\section{On endomorphisms of the semigroup $\boldsymbol{B}_{\omega}^{\mathscr{F}^3}$}

In this section for any $p\in\{0,1,2\}$ we denote $\boldsymbol{B}_{\omega}^{\mathscr{F}^3_p}=\{(i,j,[p))\colon i,j\in\omega\}$. Also we denote $\boldsymbol{B}_{\omega}^{\mathscr{F}^3_{0,1}}=\boldsymbol{B}_{\omega}^{\mathscr{F}^3_{0}}\cup\boldsymbol{B}_{\omega}^{\mathscr{F}^3_{1}}$,
$\boldsymbol{B}_{\omega}^{\mathscr{F}^3_{0,2}}=\boldsymbol{B}_{\omega}^{\mathscr{F}^3_{0}}\cup\boldsymbol{B}_{\omega}^{\mathscr{F}^3_{2}}$, and $\boldsymbol{B}_{\omega}^{\mathscr{F}^3_{1,2}}=\boldsymbol{B}_{\omega}^{\mathscr{F}^3_{1}}\cup\boldsymbol{B}_{\omega}^{\mathscr{F}^3_{2}}$.

\begin{lemma}\label{lemma-3.1}
Let $\varepsilon$ be an injective endomorphism of the semigroup $\boldsymbol{B}_{\omega}^{\mathscr{F}^3}$. Then the following statements hold:
\begin{itemize}
  \item[$(i)$]   $(\boldsymbol{B}_{\omega}^{\mathscr{F}^3})\varepsilon\nsubseteq \boldsymbol{B}_{\omega}^{\mathscr{F}^3_{0,1}}$;
  \item[$(ii)$]  $(\boldsymbol{B}_{\omega}^{\mathscr{F}^3})\varepsilon\nsubseteq \boldsymbol{B}_{\omega}^{\mathscr{F}^3_{1,2}}$;
  \item[$(iii)$] $(\boldsymbol{B}_{\omega}^{\mathscr{F}^3})\varepsilon\nsubseteq \boldsymbol{B}_{\omega}^{\mathscr{F}^3_{0,2}}$;
  \item[$(iv)$]  $(0,0,[0))\varepsilon\notin \boldsymbol{B}_{\omega}^{\mathscr{F}^3_{1}}$.
\end{itemize}
\end{lemma}

\begin{proof}
$(i)$ Suppose to the contrary that there exists an injective endomorphism $\varepsilon$ of the semigroup $\boldsymbol{B}_{\omega}^{\mathscr{F}^3}$ such that $(\boldsymbol{B}_{\omega}^{\mathscr{F}^3})\varepsilon\subseteq \boldsymbol{B}_{\omega}^{\mathscr{F}^3_{0,1}}$. If $(0,0,[0))\varepsilon\in \boldsymbol{B}_{\omega}^{\mathscr{F}^3_0}$, then by Lemma~2 of \cite{Gutik-Mykhalenych=2020} we have that $(0,0,[0))\varepsilon=(s,s,[0))$ for some $s\in\omega$. By Corollary~\ref{corollary-2.4} there exists an isomorphism $\mathfrak{I}\colon\boldsymbol{B}_{\omega}^{\mathscr{F}^{3}}(s,s,[0))\to  \boldsymbol{B}_{\omega}^{\mathscr{F}^{3}}$. Then the composition $\varepsilon\circ\mathfrak{I}$ is an injective monoid endomorphism of the semigroup $\boldsymbol{B}_{\omega}^{\mathscr{F}^3}$ such that $(\boldsymbol{B}_{\omega}^{\mathscr{F}^3})(\varepsilon\circ\mathfrak{I})\subseteq \boldsymbol{B}_{\omega}^{\mathscr{F}^3_{0,1}}\subset \boldsymbol{B}_{\omega}^{\mathscr{F}^3}$, which contradicts Theorem~5 of \cite{Gutik-Serivka=2023}. Hence $(0,0,[0))\varepsilon\notin \boldsymbol{B}_{\omega}^{\mathscr{F}^3_0}$.

If $(0,0,[0))\varepsilon\in \boldsymbol{B}_{\omega}^{\mathscr{F}^3_1}$, then by Lemma~2 of \cite{Gutik-Mykhalenych=2020} there exists $s\in\omega$ such that $(0,0,[0))\varepsilon=(s,s,[1))$. We consider the map $\varpi_2\colon \boldsymbol{B}_{\omega}^{\mathscr{F}^3_{0,1}}\to \boldsymbol{B}_{\omega}^{\mathscr{F}^3_{0,1}}\subset \boldsymbol{B}_{\omega}^{\mathscr{F}^3}$ defined by the formula
\begin{equation*}
  (i,j,[p))\varpi_2=
  \left\{
    \begin{array}{ll}
      (i,j,[1)),     & \hbox{if~} p=0;\\
      (i+1,j+1,[0)), & \hbox{if~} p=1.
    \end{array}
  \right.
\end{equation*}
By Lemma~1 of \cite{Gutik-Serivka=2025} the map $\varpi_2$ is an injective endomorphism of the subsemigroup $\boldsymbol{B}_{\omega}^{\mathscr{F}^3_{0,1}}$ of $\boldsymbol{B}_{\omega}^{\mathscr{F}^3}$, because by Corollary~\ref{corollary-2.4} the semigroups $\boldsymbol{B}_{\omega}^{\mathscr{F}^3_{0,1}}$ and $\boldsymbol{B}_{\omega}^{\mathscr{F}^2}$ are isomorphic. Then the composition
 ${\varepsilon\circ\varpi_2}\colon \boldsymbol{B}_{\omega}^{\mathscr{F}^3}\to \boldsymbol{B}_{\omega}^{\mathscr{F}^3_{0,1}}\subset \boldsymbol{B}_{\omega}^{\mathscr{F}^3}$ is an injective endomorphism of the semigroup $\boldsymbol{B}_{\omega}^{\mathscr{F}^3}$ such that $(0,0,[0))\varepsilon\in \boldsymbol{B}_{\omega}^{\mathscr{F}^3_0}$, which contradicts the above part of the proof of this statement. The obtained contradiction implies statement~$(i)$.

$(ii)$  Suppose to the contrary that there exists an injective endomorphism $\varepsilon$ of the semigroup $\boldsymbol{B}_{\omega}^{\mathscr{F}^3}$ such that $(\boldsymbol{B}_{\omega}^{\mathscr{F}^3})\varepsilon\subseteq \boldsymbol{B}_{\omega}^{\mathscr{F}^3_{1,2}}$. By Corollary~\ref{corollary-2.4} there exists an isomorphism $\mathfrak{I}\colon\boldsymbol{B}_{\omega}^{\mathscr{F}^3_{1,2}}\to\boldsymbol{B}_{\omega}^{\mathscr{F}^3_{0,1}}$. Then the composition $\varepsilon\circ\mathfrak{I}$ is an injective endomorphism of $\boldsymbol{B}_{\omega}^{\mathscr{F}^3}$ such that $(\boldsymbol{B}_{\omega}^{\mathscr{F}^3})(\varepsilon\circ\mathfrak{I})\subseteq \boldsymbol{B}_{\omega}^{\mathscr{F}^3_{1,2}}$, which contradicts statement $(i)$. The obtained contradiction implies statement~$(ii)$.

$(iii)$ Suppose the contrary: there exists an injective endomorphism $\varepsilon$ of the semigroup $\boldsymbol{B}_{\omega}^{\mathscr{F}^3}$ such that $(\boldsymbol{B}_{\omega}^{\mathscr{F}^3})\varepsilon\subseteq \boldsymbol{B}_{\omega}^{\mathscr{F}^3_{0,2}}$. If $(0,0,[0))\varepsilon\in \boldsymbol{B}_{\omega}^{\mathscr{F}^3_0}$, then by Lemma~2 of \cite{Gutik-Mykhalenych=2020} we get that $(0,0,[0))\varepsilon=(s,s,[0))$ for some $s\in\omega$. Since $(0,0,[0))$ is the identity element of the monoid $\boldsymbol{B}_{\omega}^{\mathscr{F}^3}$, $(0,0,[0))\varepsilon$ is the identity element of the subsemigroup $(\boldsymbol{B}_{\omega}^{\mathscr{F}^3})\varepsilon$ of $\boldsymbol{B}_{\omega}^{\mathscr{F}^3}$. This implies that $(\boldsymbol{B}_{\omega}^{\mathscr{F}^3})\varepsilon$ is a subsemigroup of the homomorphic image $(\boldsymbol{B}_{\omega}^{\mathscr{F}^3})\lambda^s$. By Corollary~\ref{corollary-2.4} there exists an isomorphism $\mathfrak{I}\colon(\boldsymbol{B}_{\omega}^{\mathscr{F}^3_{1,2}})\lambda^s\to\boldsymbol{B}_{\omega}^{\mathscr{F}^3}$. This implies that the composition $\varepsilon\circ\mathfrak{I}$ is an injective monoid endomorphism of $\boldsymbol{B}_{\omega}^{\mathscr{F}^3}$. Then by Theorem~5 of \cite{Gutik-Serivka=2023} there exists a positive integer $k$ such that $\alpha_{[k]}=\varepsilon\circ\mathfrak{I}$. Then $\varepsilon=\alpha_{[k]}\circ\mathfrak{I}^{-1}$ and $(\boldsymbol{B}_{\omega}^{\mathscr{F}^3})\varepsilon\nsubseteq \boldsymbol{B}_{\omega}^{\mathscr{F}^3_{0,2}}$.

If $(0,0,[0))\varepsilon\in \boldsymbol{B}_{\omega}^{\mathscr{F}^3_2}$, then by Lemma~2 of \cite{Gutik-Mykhalenych=2020} there exists $s\in\omega$ such that  $(0,0,[0))\varepsilon=(s,s,[0))$. Then the composition $\varepsilon\circ\varpi_3$ is an injective endomorphism of the semigroup $\boldsymbol{B}_{\omega}^{\mathscr{F}^3}$ such that $(0,0,[0))\varepsilon\circ\varpi_3\in \boldsymbol{B}_{\omega}^{\mathscr{F}^3_0}$. This contradicts the above part of proof of this statement. The obtained contradiction implies statement~$(iii)$.

$(iv)$ Suppose to the contrary that there exists an injective endomorphism $\varepsilon$ of the semigroup $\boldsymbol{B}_{\omega}^{\mathscr{F}^3}$ such that $(0,0,[0))\varepsilon\in \boldsymbol{B}_{\omega}^{\mathscr{F}^3_{1}}$. By statements $(i)$ and $(ii)$ neither the condition $(0,0,[1))\varepsilon\in \boldsymbol{B}_{\omega}^{\mathscr{F}^3_{1}}$ nor the condition $(0,0,[2))\varepsilon\in \boldsymbol{B}_{\omega}^{\mathscr{F}^3_{1}}$ holds.

If $(0,0,[1))\varepsilon\in \boldsymbol{B}_{\omega}^{\mathscr{F}^3_{2}}$, then by statement $(ii)$ we have that $(0,0,[2))\varepsilon\in \boldsymbol{B}_{\omega}^{\mathscr{F}^3_{0}}$. By Lemma~2 of \cite{Gutik-Mykhalenych=2020} we have that $(0,0,[0))\varepsilon=(s,s,[1))$ for some $s\in\omega$. Proposition~\ref{proposition-2.1} and Corollary~\ref{corollary-2.4} imply that the corestriction of the mapping $\lambda^s\colon \boldsymbol{B}_{\omega}^{\mathscr{F}^3}\to \boldsymbol{B}_{\omega}^{\mathscr{F}^3}$ onto the image $(\boldsymbol{B}_{\omega}^{\mathscr{F}^3})\lambda^s$ is an isomorphism of $\boldsymbol{B}_{\omega}^{\mathscr{F}^3}$ onto $(\boldsymbol{B}_{\omega}^{\mathscr{F}^3})\lambda^s$, and we denote this isomorphism by $\mathfrak{J}$. Then the composition $\varepsilon\circ\mathfrak{J}^{-1}$ is a monoid endomorphism of the semigroup $\boldsymbol{B}_{\omega}^{\mathscr{F}^3}$. Hence without loss of generality later we may assume that $(0,0,[0))\varepsilon=(0,0,[1))$.

Since by Corollary~\ref{corollary-2.4} the subsemigroups $\boldsymbol{B}_{\omega}^{\mathscr{F}^3_{0,1}}$ and $\boldsymbol{B}_{\omega}^{\mathscr{F}^3_{1,2}}$ of $\boldsymbol{B}_{\omega}^{\mathscr{F}^3}$ are isomorphic under the mapping $\mathfrak{J}\colon (i,j,[p))\mapsto i,j,[p+1))$, $i,j\in\omega$, $p\in\{0,1\}$, Theorem~1 of \cite{Gutik-Pozdniakova=2022} implies that there exist a positive integer $k$ and $p\in\{0,1\}$ such that
\begin{align*}
  (i,j,[0))\varepsilon &=(ki,kj,[0)); \\
  (i,j,[1))\varepsilon &=(p+ki,p+kj,[1)),
\end{align*}
for all $i,j\in\omega$.

By Lemma~5 of \cite{Gutik-Mykhalenych=2020} we have that
\begin{equation*}
  (0,0,[2))\preccurlyeq (0,0,[1))\preccurlyeq (0,0,[0)) \qquad \hbox{and} \qquad (2,2,[0))\preccurlyeq (1,1,[1))\preccurlyeq (0,0,[2)).
\end{equation*}
By Proposition~1.4.21$(6)$ from \cite{Lawson=1998} we get that
\begin{equation*}
  (0,0,[2))\varepsilon\preccurlyeq (0,0,[1))\varepsilon=(p,p,[2))\preccurlyeq (0,0,[0))\varepsilon=(0,0,[1))
\end{equation*}
and
\begin{equation*}
  (2,2,[0))\varepsilon=(2k,2k,[1))\preccurlyeq (1,1,[1))\varepsilon(k+p,k+p,[2))\preccurlyeq (0,0,[2))\varepsilon.
\end{equation*}
Since $(0,0,[2))\varepsilon\preccurlyeq (0,0,[1))$ Lemma~2 of \cite{Gutik-Mykhalenych=2020} implies that there exists a positive integer  $x$ such that  $(0,0,[2))\varepsilon=(x,x,[0))$. By the equality
\begin{equation*}
(1,1,[0))\cdot(0,0,[2))=(1,1,[1))
\end{equation*}
we have that
\begin{equation*}
(1,1,[0))\varepsilon\cdot(0,0,[2))\varepsilon=(1,1,[1))\varepsilon.
\end{equation*}
Hence we get the equality
\begin{equation*}
  (k,k,[1))\cdot (x,x,[0))=(k+p,k+p,[2)),
\end{equation*}
which contradicts the semigroup operation on $\boldsymbol{B}_{\omega}^{\mathscr{F}^3}$. The obtained contradiction implies that the conditions $(0,0,[0))\varepsilon\in \boldsymbol{B}_{\omega}^{\mathscr{F}^3_1}$, $(0,0,[1))\varepsilon\in \boldsymbol{B}_{\omega}^{\mathscr{F}^3_2}$, and $(0,0,[2))\varepsilon\in \boldsymbol{B}_{\omega}^{\mathscr{F}^3_0}$ do not hold.

Suppose there exists an injective endomorphism $\varepsilon$ of the semigroup $\boldsymbol{B}_{\omega}^{\mathscr{F}^3}$ such that $(0,0,[0))\varepsilon\in \boldsymbol{B}_{\omega}^{\mathscr{F}^3_{1}}$, $(0,0,[1))\varepsilon\in \boldsymbol{B}_{\omega}^{\mathscr{F}^3_0}$, and $(0,0,[2))\varepsilon\in \boldsymbol{B}_{\omega}^{\mathscr{F}^3_2}$. By Proposition~\ref{proposition-2.1}$(ii)$ $\varpi_3$ is an injective endomorphism of the semigroup $\boldsymbol{B}_{\omega}^{\mathscr{F}^3}$ and hence the composition $\varepsilon\circ\varpi_3$ is an injective endomorphism of $\boldsymbol{B}_{\omega}^{\mathscr{F}^3}$, as well. Then we get that $(0,0,[0))(\varepsilon\circ\varpi_3)\in \boldsymbol{B}_{\omega}^{\mathscr{F}^3_1}$, $(0,0,[1))(\varepsilon\circ\varpi_3)\in \boldsymbol{B}_{\omega}^{\mathscr{F}^3_2}$, and $(0,0,[2))(\varepsilon\circ\varpi_3)\in \boldsymbol{B}_{\omega}^{\mathscr{F}^3_0}$, which contradicts the above part of the proof.

Therefore we obtain that $(0,0,[0))\varepsilon\notin \boldsymbol{B}_{\omega}^{\mathscr{F}^3_{1}}$.
\end{proof}

If $S$ is a semigroup and $a,b\in S$, then by $\langle a,b\rangle$ we denote the subsemigroup of $S$ which is generated by elements $a$ and $b$.

\begin{theorem}\label{theorem-3.2}
For every injective endomorphism $\varepsilon$ of the semigroup $\boldsymbol{B}_{\omega}^{\mathscr{F}^3}$ there exists an injective endomorphism $\iota\in\left\langle\lambda,\varpi_3\right\rangle$ such that $\varepsilon=\alpha_{[k]}\circ\iota$ for some positive integer $k$.
\end{theorem}

\begin{proof}
Lemma~\ref{lemma-3.1} implies that for every endomorphism of the semigroup $\boldsymbol{B}_{\omega}^{\mathscr{F}^3}$ only one of the following conditions holds:
\begin{itemize}
  \item[(1)] $(0,0,[p))\varepsilon\in \boldsymbol{B}_{\omega}^{\mathscr{F}^3_p}$ for all $p\in\{0,1,2\}$;
  \item[(2)] $(0,0,[p))\varepsilon\in \boldsymbol{B}_{\omega}^{\mathscr{F}^3_{2-p}}$ for all $p\in\{0,1,2\}$.
\end{itemize}

Suppose condition (1) holds. By Proposition~3 \cite{Gutik-Mykhalenych=2020} for any $p\in\{0,1,2\}$ the subsemigroup $\boldsymbol{B}_{\omega}^{\mathscr{F}^3_p}$ of $\boldsymbol{B}_{\omega}^{\mathscr{F}^3}$ is isomorphic to the bicyclic semigroup. Hence by Proposition~4 of  \cite{Gutik-Mykhalenych=2021} we have that $(i,j,[p))\varepsilon\in \boldsymbol{B}_{\omega}^{\mathscr{F}^3_p}$ for all $i,j\in\omega$ and $p\in\{0,1,2\}$. Since $(0,0,[0))$ is an idempotent of $\boldsymbol{B}_{\omega}^{\mathscr{F}^3}$ Lemma~2 of \cite{Gutik-Mykhalenych=2020} implies that there exists $s\in\omega$ such that  $(0,0,[0))\varepsilon=(s,s,[0))$. This implies that $(\boldsymbol{B}_{\omega}^{\mathscr{F}^3})\varepsilon\subseteq (\boldsymbol{B}_{\omega}^{\mathscr{F}^3})\lambda^s$. By Proposition~\ref{proposition-2.1} and Corollary~\ref{corollary-2.4} the corestriction $\lambda^s{\upharpoonright}_{(\boldsymbol{B}_{\omega}^{\mathscr{F}^3})\lambda^s}\colon \boldsymbol{B}_{\omega}^{\mathscr{F}^3}\to (\boldsymbol{B}_{\omega}^{\mathscr{F}^3})\lambda^s$ of the endomorphism $\lambda^s$ onto $(\boldsymbol{B}_{\omega}^{\mathscr{F}^3})\lambda^s$ is an isomorphism, and hence we get that the composition $\varepsilon\circ(\lambda^s{\upharpoonright}_{(\boldsymbol{B}_{\omega}^{\mathscr{F}^3})\lambda^s})^{-1}$ is an injective monoid endomorphism of the semigroup $\boldsymbol{B}_{\omega}^{\mathscr{F}^3}$. This implies that there exists a positive integer $k$ such that $\varepsilon\circ(\lambda^s{\upharpoonright}_{(\boldsymbol{B}_{\omega}^{\mathscr{F}^3})\lambda^s})^{-1}=\alpha_{[k]}$. The injectivity of the mapping $\lambda^s$ implies that $\varepsilon=\alpha_{[k]}\circ\lambda^s{\upharpoonright}_{(\boldsymbol{B}_{\omega}^{\mathscr{F}^3})\lambda^s}=\alpha_{[k]}\circ\lambda^s$.

Suppose condition (2) holds. Then $\varepsilon\circ\varpi_3$ is an injective endomorphism of the semigroup $\boldsymbol{B}_{\omega}^{\mathscr{F}^3}$ which satisfies condition (1).  By the previous part of the proof there exists a positive integer $k$ such that $\varepsilon\circ\varpi_3=\alpha_{[k]}\circ\lambda^s$. We remark that the above equality holds if and only if $\varepsilon\circ\varpi_3\circ\varpi_3=\alpha_{[k]}\circ\lambda^s\circ\varpi_3$, because $\varepsilon$, $\varpi_3$, $\alpha_{[k]}$, and $\lambda^s$ are injective maps. By Proposition~\ref{proposition-2.1}$(iii)$ we get that
\begin{equation}\label{eq-3.1}
  \varepsilon\circ\lambda^{k-1}=\varepsilon\circ\varpi_3\circ\varpi_3=\alpha_{[k]}\circ\lambda^s\circ\varpi_3.
\end{equation}
Suppose $(0,0,[0))\varepsilon=(x,x,[2))$ for some $x\in\omega$. Then we have that
\begin{align*}
  (0,0,[0))(\varepsilon\circ\lambda^{k-1})&=(x,x,[2))\lambda^{k-1}= \\
   &=(x+k-1,x+k-1,[2))
\end{align*}
and
\begin{align*}
  (0,0,[0))(\alpha_{[k]}\circ\lambda^s\circ\varpi_3)&=(0,0,[0))(\lambda^s\circ\varpi_3)= \\
   &=(s,s,[0))\varpi_3=\\
   &=(s,s,[2)).
\end{align*}
This implies that $k-1\leqslant s$. Then equality \eqref{eq-3.1}, Proposition~\ref{proposition-2.1},and the injectivity of above maps imply that
\begin{align*}
  \varepsilon=\alpha_{[k]}\circ\lambda^s\circ\varpi_3\circ(\lambda^{k-1})^{-1}&=\alpha_{[k]}\circ\varpi_3\circ\lambda^s\circ(\lambda^{k-1})^{-1}= \\
   &=\alpha_{[k]}\circ\varpi_3\circ\lambda^{s-k+1}=\\
   &=\alpha_{[k]}\circ\lambda^{s-k+1}\circ\varpi_3.
\end{align*}
This completes the proof of the theorem.
\end{proof}

\begin{lemma}\label{lemma-3.3}
Let $X$ be a non-empty set and $\mathfrak{a}$ and $\mathfrak{b}$ be transformations of the set $X$.  Then for an arbitrary injective transformation $\mathfrak{h}$ of $X$ the equality $\mathfrak{a}\circ\mathfrak{h}=\mathfrak{b}\circ\mathfrak{h}$ implies that $\mathfrak{a}=\mathfrak{b}$.
\end{lemma}

\begin{proof}
Suppose to the contrary that there exist transformation $\mathfrak{a}$ and $\mathfrak{b}$ of the set $X$ such that $\mathfrak{a}\circ\mathfrak{h}=\mathfrak{b}\circ\mathfrak{h}$ and $\mathfrak{a}\neq\mathfrak{b}$. Then there exists $x\in X$ such that $(x)\mathfrak{a}\neq(x)\mathfrak{b}$. Since $\mathfrak{h}$ is an injective transformation of $X$, we get that $(x)(\mathfrak{a}\circ\mathfrak{h})\neq(x)(\mathfrak{b}\circ\mathfrak{h})$, which contradicts the assumption. The obtained contradiction implies the statement of the lemma.
\end{proof}

By $\boldsymbol{End}_{\textsf{inj}}(\boldsymbol{B}_{\omega}^{\mathscr{F}^3})$ we denote the semigroup of all injective endomorphisms of the semigroup $\boldsymbol{B}_{\omega}^{\mathscr{F}^3}$.

\begin{proposition}\label{proposition-3.4}
Let $k$ be an arbitrary positive integer. Then in $\boldsymbol{End}_{\textsf{inj}}(\boldsymbol{B}_{\omega}^{\mathscr{F}^3})$ the following equalities hold:
\begin{itemize}
  \item[$(i)$]   $\lambda\circ\alpha_{[k]}=\alpha_{[k]}\circ\lambda^{k}$;
  \item[$(ii)$]  $\varpi_3\circ\alpha_{[k]}\circ\varpi_3=\alpha_{[k]}\circ\lambda^{k+1}$;
  \item[$(iii)$] $\varpi_3\circ\alpha_{[k]}=\alpha_{[k]}\circ\varpi_3\circ\lambda^{k-1}$.
\end{itemize}
\end{proposition}

\begin{proof}
$(i)$ For any $i,j\in\omega$ and $p\in\{0,1,2\}$ we have that
\begin{align*}
  (i,j,[p))(\lambda\circ\alpha_{[k]})&= (i+1,j+1,[p))\alpha_{[k]}=\\
  &=
   \left\{
    \begin{array}{ll}
      (k(i+1),k(j+1),[p)),     & \hbox{if~} p\in\{0,1\};\\
      (k(i+2)-1,k(j+2)-1,[2)), & \hbox{if~} p=2
    \end{array}
  \right.=
  \\
   &=
   \left\{
    \begin{array}{ll}
      (ki+k,kj+k,[p)),             & \hbox{if~} p\in\{0,1\};\\
      (k(i+1)-1+k,k(j+1)-1+k,[2)), & \hbox{if~} p=2
    \end{array}
  \right.
\end{align*}
and
\begin{align*}
  (i,j,[p))(\alpha_{[k]}\circ\lambda^{k})&=
   \left\{
    \begin{array}{ll}
      (ki,kj,[p))\lambda^{k},             & \hbox{if~} p\in\{0,1\};\\
      (k(i+1)-1,k(j+1)-1,[2))\lambda^{k}, & \hbox{if~} p=2
    \end{array}
  \right.= \\
   &=
    \left\{
    \begin{array}{ll}
      (ki+k,kj+k,[p)),             & \hbox{if~} p\in\{0,1\};\\
      (k(i+1)-1+k,k(j+1)-1+k,[2)), & \hbox{if~} p=2.
    \end{array}
  \right.
\end{align*}
This implies that the equality $\lambda\circ\alpha_{[k]}=\alpha_{[k]}\circ\lambda^{k}$ holds in $\boldsymbol{End}_{\textsf{inj}}(\boldsymbol{B}_{\omega}^{\mathscr{F}^3})$ .

\smallskip

$(ii)$ For any $i,j\in\omega$ and $p\in\{0,1,2\}$ we have that
\begin{align*}
  (i,j,[p))(\varpi_3\circ\alpha_{[k]}\circ\varpi_3)&=(i+p,j+p,[2-p))(\alpha_{[k]}\circ\varpi_3)= \\
   &=
   \left\{
     \begin{array}{ll}
       (i,j,[2))(\alpha_{[k]}\circ\varpi_3),     & \hbox{if~} p=0;\\
       (i+1,j+1,[1))(\alpha_{[k]}\circ\varpi_3), & \hbox{if~} p=1;\\
       (i+2,j+2,[0))(\alpha_{[k]}\circ\varpi_3), & \hbox{if~} p=2
     \end{array}
   \right.= \\
\end{align*}

\begin{align*}
   &=
   \left\{
     \begin{array}{ll}
       (k(i+1)-1,k(j+1)-1,[2))\varpi_3, & \hbox{if~} p=0;\\
       (k(i+1),k(j+1),[1))\varpi_3,     & \hbox{if~} p=1;\\
       (k(i+2),k(j+2),[0))\varpi_3,     & \hbox{if~} p=2
     \end{array}
   \right.= \\
   &=
   \left\{
     \begin{array}{ll}
       (k(i+1)-1+2,k(j+1)-1+2,[2-2)), & \hbox{if~} p=0;\\
       (k(i+1)+1,k(j+1)+1,[2-1)),     & \hbox{if~} p=1;\\
       (k(i+2),k(j+2),[2-0)),         & \hbox{if~} p=2
     \end{array}
   \right.= \\
   &=
   \left\{
     \begin{array}{ll}
       (ki+k+1,kj+k+1,[0)),     & \hbox{if~} p=0;\\
       (ki+k+1,kj+k+1,[1)),     & \hbox{if~} p=1;\\
       (k(i+1)+k,k(j+1)+k,[2)), & \hbox{if~} p=2
     \end{array}
   \right.
\end{align*}
and
\begin{align*}
  (i,j,[p))(\alpha_{[k]}\circ\lambda^{k+1})&=
    \left\{
      \begin{array}{ll}
        (ki,kj,[p))\lambda^{k+1},             & \hbox{if~} p\in\{0,1\}; \\
        (k(i+1)-1,k(j+1)-1,[2))\lambda^{k+1}, & \hbox{if~} p=2
      \end{array}
    \right.=
    \\
   &=\left\{
      \begin{array}{ll}
        (ki+k+1,kj+k+1,[p)),             & \hbox{if~} p\in\{0,1\}; \\
        (k(i+1)-1+k+1,k(j+1)-1+k+1,[2)), & \hbox{if~} p=2
      \end{array}
    \right.=
    \\
   &=\left\{
      \begin{array}{ll}
        (ki+k+1,kj+k+1,[p)),     & \hbox{if~} p\in\{0,1\}; \\
        (k(i+1)+k,k(j+1)+k,[2)), & \hbox{if~} p=2
      \end{array}
    \right.,
\end{align*}
and hence $\varpi_3\circ\alpha_{[k]}\circ\varpi_3=\alpha_{[k]}\circ\lambda^{k+1}$ in $\boldsymbol{End}_{\textsf{inj}}(\boldsymbol{B}_{\omega}^{\mathscr{F}^3})$.

\smallskip

$(iii)$ The equality $\varpi_3\circ\alpha_{[k]}\circ\varpi_3=\alpha_{[k]}\circ\lambda^{k+1}$ implies that $\varpi_3\circ\alpha_{[k]}\circ\varpi_3\circ\varpi_3=\alpha_{[k]}\circ\lambda^{k+1}\circ\varpi_3$. By Proposition~\ref{proposition-2.1}$(iii)$ we get that $\varpi_3^2=\lambda^2$, and hence $\varpi_3\circ\alpha_{[k]}\circ\lambda^2=\alpha_{[k]}\circ\lambda^{k+1}\circ\varpi_3$. Since $k$ is a positive integer and $\lambda^2$ is an injective endomorphism of the semigroup $\boldsymbol{B}_{\omega}^{\mathscr{F}^3}$, Lemma~\ref{lemma-3.3}, Proposition~\ref{proposition-2.1}$(iv)$, and the equality $\varpi_3\circ\alpha_{[k]}\circ\lambda^2=\alpha_{[k]}\circ\lambda^{k+1}\circ\varpi_3$ imply that $\varpi_3\circ\alpha_{[k]}=\alpha_{[k]}\circ\varpi_3\circ\lambda^{k-1}$.
\end{proof}

\begin{remark}\label{remark-3.5}
By Theorem~5 of \cite{Gutik-Serivka=2023} for any injective monoid endomorphism $\varepsilon$ of the semigroup $\boldsymbol{B}_{\omega}^{\mathscr{F}^3}$ there exists a positive integer $k$ such that $\varepsilon=\alpha_{[k]}$. Also (see: \cite{Gutik-Serivka=2023}), we have that  $\alpha_{[k_1]}\circ\alpha_{[k_2]}=\alpha_{[k_1\cdot k_2]}$ in the monoid $\boldsymbol{End}_{\textsf{inj}}^1(\boldsymbol{B}_{\omega}^{\mathscr{F}^3})$ of all injective monoid endomorphisms of  $\boldsymbol{B}_{\omega}^{\mathscr{F}^3}$ for all positive integers $k_1$ and $k_2$.
\end{remark}

By $\left\langle\alpha_{[\bullet]}\right\rangle$ we denote the subsemigroup of $\boldsymbol{End}_{\textsf{inj}}(\boldsymbol{B}_{\omega}^{\mathscr{F}^3})$ which is generated by endomorphisms $\alpha_{[k]}$, $k\in\mathbb{N}$ of the semigroup $\boldsymbol{B}_{\omega}^{\mathscr{F}^3}$. Theorem~6 of \cite{Gutik-Serivka=2023} states that the semigroup $\left\langle\alpha_{[\bullet]}\right\rangle$ is isomorphic to the multiplicative monoid of positive integers $(\mathbb{N},\cdot)$.

Later without loss of generality we may assume that $\lambda^0$ is an identity endomorphism of the semigroup $\boldsymbol{B}_{\omega}^{\mathscr{F}^3}$. It is obvious that the semigroup $\lambda^+=\{\lambda^m\colon m\in\mathbb{N}\}$ is isomorphic to the additive semigroup of positive integers $(\mathbb{N},+)$ and the monoid $\lambda^*=\{\lambda^m\colon m\in\omega\}$ is isomorphic to the additive monoid of non-negative integers $(\omega,+)$.

Let $S$ and $T$ be semigroups and $\mathfrak{h}\colon S\to \boldsymbol{End}(T)$ be a homomorphism from $S$ into the semigroup $\boldsymbol{End}(T)$ of all endomorphisms of $T$. By $\mathfrak{h}_{s_2}$ we denote the image of $s_2$ under homomorphism $\mathfrak{h}$. The Cartesian product $S\times T$ with the following semigroup operation
\begin{equation*}
  (s_1,t_1)\cdot (s_2,t_2)=(s_1s_2,(t_1)\mathfrak{h}_{s_2} t_2)
\end{equation*}
is called  the \emph{semidirect product} $S$ and $T$, and it is denoted by $S\ltimes_{\mathfrak{h}}T$.

\begin{theorem}\label{theorem-3.6}
In the monoid $\boldsymbol{End}_{\textsf{inj}}(\boldsymbol{B}_{\omega}^{\mathscr{F}^3})$ the following statements hold:
\begin{itemize}
  \item[$(i)$] $(\alpha_{[k_1]}\circ\lambda^{m_1})\circ(\alpha_{[k_2]}\circ\lambda^{m_2})=\alpha_{[k_1\cdot k_2]}\circ\lambda^{k_2\cdot m_1+m_2}$ for any $k_1,k_2\in\mathbb{N}$ and $m_1,m_2\in\omega$;
  \item[$(ii)$] the submonoid $\left\langle\left\langle\alpha_{[\bullet]}\right\rangle,\lambda^*\right\rangle$ of $\boldsymbol{End}_{\textsf{inj}}(\boldsymbol{B}_{\omega}^{\mathscr{F}^3})$, which is generated by the sets $\left\langle\alpha_{[\bullet]}\right\rangle$ and $\lambda^*$, is isomorphic to the semidirect product $(\mathbb{N},\cdot)\ltimes_{\mathfrak{h}}(\omega,+)$, where the homomorphisma $\mathfrak{h}\colon (\mathbb{N},\cdot)\to \boldsymbol{End}(\omega,+)$ is defined by the formula $(n)\mathfrak{h}_k=kn$;
  \item[$(iii)$] $(\alpha_{[k_1]}\circ\lambda^{m_1}\circ\varpi_3)\circ(\alpha_{[k_2]}\circ\lambda^{m_2})=\alpha_{[k_1\cdot k_2]}\circ\lambda^{k_2\cdot m_1+k_2+m_2}\circ\varpi_3$ for any $k_1,k_2\in\mathbb{N}$ and $m_1,m_2\in\omega$;
  \item[$(iv)$] $(\alpha_{[k_1]}\circ\lambda^{m_1}\circ\varpi_3)\circ(\alpha_{[k_2]}\circ\lambda^{m_2}\circ\varpi_3)=\alpha_{[k_1\cdot k_2]}\circ\lambda^{k_2\cdot m_1+k_2+m_2+2}$ for any $k_1,k_2\in\mathbb{N}$ and $m_1,m_2\in\omega$.
\end{itemize}
\end{theorem}

\begin{proof}
$(i)$ By Proposition~\ref{proposition-3.4}$(i)$ for any $k_1,k_2\in\mathbb{N}$ and $m_1,m_2\in\omega$ we have that
\begin{align*}
  \alpha_{[k_1]}\circ\lambda^{m_1}\circ\alpha_{[k_2]}\circ\lambda^{m_2}&=
   \alpha_{[k_1]}\circ\alpha_{[k_2]}\circ\lambda^{k_2\cdot m_1}\circ\lambda^{m_2}=\\
   &=\alpha_{[k_1\cdot k_2]}\circ\lambda^{k_2\cdot m_1+m_2}.
\end{align*}

$(ii)$ By Proposition~\ref{proposition-3.4}$(i)$ every element of the monoid $\left\langle\left\langle\alpha_{[\bullet]}\right\rangle,\lambda^*\right\rangle$ is presented as a composition $\alpha_{[k]}\circ\lambda^m$ for some $k\in\mathbb{N}$ and $m\in\omega$. It is obvious that such a presentation is unique. We define an isomorphism $\mathfrak{I}\colon \left\langle\left\langle\alpha_{[\bullet]}\right\rangle,\lambda^*\right\rangle\to (\mathbb{N},\cdot)\ltimes_{\mathfrak{h}}(\omega,+)$ by the formula $(\alpha_{[k]}\circ\lambda^m)\mathfrak{I}=(k,m)$ for all $k\in\mathbb{N}$ and $m\in\omega$. Statement $(i)$ implies that so defined map $\mathfrak{I}$ is an isomorphisms of the semigroups $\left\langle\left\langle\alpha_{[\bullet]}\right\rangle,\lambda^*\right\rangle$ and $(\mathbb{N},\cdot)\ltimes_{\mathfrak{h}}(\omega,+)$.

\smallskip

$(iii)$ By Proposition~\ref{proposition-3.4} and statement $(i)$ for any $k_1,k_2\in\mathbb{N}$ and $m_1,m_2\in\omega$ we have that
\begin{align*}
  (\alpha_{[k_1]}\circ\lambda^{m_1}\circ\varpi_3)\circ(\alpha_{[k_2]}\circ\lambda^{m_2})&=
    \alpha_{[k_1]}\circ\lambda^{m_1}\circ\alpha_{[k_2]}\circ\varpi_3\circ\lambda^{k_2}\circ\lambda^{m_2}=\\
   &=\alpha_{[k_1]}\circ\alpha_{[k_2]}\circ\lambda^{k_2\cdot m_1}\circ\varpi_3\circ\lambda^{k_2+m_2}=\\
   &=\alpha_{[k_1\cdot k_2]}\circ\lambda^{k_2\cdot m_1+k_2+m_2}\circ\varpi_3.
\end{align*}

$(iv)$ By statement $(ii)$ and Proposition~\ref{proposition-2.1}$(iii)$ for any $k_1,k_2\in\mathbb{N}$ and $m_1,m_2\in\omega$ we get that
\begin{align*}
  (\alpha_{[k_1]}\circ\lambda^{m_1}\circ\varpi_3)\circ(\alpha_{[k_2]}\circ\lambda^{m_2}\circ\varpi_3)&=
    \alpha_{[k_1\cdot k_2]}\circ\lambda^{k_2\cdot m_1+k_2+m_2}\circ\varpi_3^2=\\
   &=\alpha_{[k_1\cdot k_2]}\circ\lambda^{k_2\cdot m_1+k_2+m_2}\circ\lambda^2=\\
   &=\alpha_{[k_1\cdot k_2]}\circ\lambda^{k_2\cdot m_1+k_2+m_2+2}.
\end{align*}
\end{proof}
\section*{\textbf{Acknowledgements}}

The authors acknowledge Alex Ravsky for his  comments and suggestions.




\begin{thebibliography}{11}


\bibitem{Clifford-Preston-1961}
A. H.~Clifford and  G. B.~Preston,
\emph{The algebraic theory of semigroups},
Vol. I., Amer. Math. Soc. Surveys 7, Pro\-vi\-den\-ce, R.I., 1961.

\bibitem{Clifford-Preston-1967}
A. H.~Clifford and  G. B.~Preston,
\emph{The algebraic theory of semigroups},
Vol. II., Amer. Math. Soc. Surveys 7, Provi\-den\-ce, R.I., 1967.
 
\bibitem{Gutik-Lysetska=2021}
O. Gutik and O. Lysetska,
\emph{On the semigroup $\boldsymbol{B}_{\omega}^{\mathscr{F}}$ which is generated by the family $\mathscr{F}$ of atomic subsets of $\omega$},
Visn. L'viv. Univ., Ser. Mekh.-Mat. \textbf{92} (2021), 34--50.
DOI: 10.30970/vmm.2021.92.034-050


\bibitem{Gutik-Mykhalenych=2020}
O. Gutik and M. Mykhalenych,
\emph{On some generalization of the bicyclic monoid},
Visnyk Lviv. Univ. Ser. Mech.-Mat. \textbf{90} (2020), 5--19 (in Ukrainian).
DOI: 10.30970/vmm.2020.90.005-019


\bibitem{Gutik-Mykhalenych=2021}
O. Gutik and M. Mykhalenych,
\emph{On group congruences on the semigroup $\boldsymbol{B}_{\omega}^{\mathscr{F}}$ and its homo\-mor\-phic retracts in the case when the family $\mathscr{F}$ consists of inductive non-empty subsets of~$\omega$},
Visnyk Lviv. Univ. Ser. Mech.-Mat. \textbf{91} (2021), 5--27 (in Ukrainian). 
DOI:  10.30970/vmm.2021.91.005-027

\bibitem{Gutik-Mykhalenych=2022}
O. Gutik and M. Mykhalenych,
\emph{On automorphisms of the semigroup $\boldsymbol{B}_{\omega}^{\mathscr{F}}$ in the case when the family $\mathscr{F}$ consists of nonempty inductive subsets of ${\omega}$},
Visnyk Lviv. Univ. Ser. Mech.-Mat. \textbf{93} (2022), 54--65 (in Ukrainian).
DOI: 10.30970/vmm.2022.93.054-065



\bibitem{Gutik-Popadiuk=2023}
O. V. Gutik and O. B. Popadiuk,
\emph{On the semigroup $\boldsymbol{B}_{\omega}^{\mathscr{F}_n}$, which is generated by the family $\mathscr{F}_n$ of finite bounded intervals of $\omega$},
Carpathian Math. Publ. \textbf{15} (2023), no. 2, 331--355. 
DOI: 0.15330/cmp.15.2.331-355

\bibitem{Gutik-Popadiuk=2022}
O. V. Gutik and O. B. Popadiuk,
\emph{On the semigroup of injective endomorphisms of the semigroup $\boldsymbol{B}_{\omega}^{\mathscr{F}_n}$ which is generated by the family $\mathscr{F}_n$ of initial finite intervals of $\omega$},
J. Math. Sci.  \textbf{282} (2024), no. 5, 646--667.\\
DOI: 10.1007/s10958-024-07207-9


\bibitem{Gutik-Pozdniakova=2022}
O. Gutik and I. Pozdniakova,
\emph{On the semigroup of injective monoid endomorphisms of the monoid $\boldsymbol{B}_{\omega}^{\mathscr{F}}$ with the two-elements family $\mathscr{F}$ of inductive nonempty subsets of $\omega$},
Visnyk Lviv. Univ. Ser. Mech.-Mat. \textbf{94} (2022), 32--55.
DOI: 10.30970/vmm.2022.94.032-055

\bibitem{Gutik-Pozdniakova=2023a}
O. Gutik and I. Pozdniakova,
\emph{On the semigroup of non-injective monoid endomorphisms of the semigroup $\boldsymbol{B}_{\omega}^{\mathscr{F}}$ with the two-element family $\mathscr{F}$ of inductive nonempty subsets of $\omega$},
Visnyk Lviv. Univ. Ser. Mech.-Mat. \textbf{95} (2023), 14--27.
DOI: 10.30970/vmm.2023.95.014-027

\bibitem{Gutik-Pozdniakova=2023c}
O. Gutik and I. Pozdniakova,
\emph{On the group of automorphisms of the semigroup $\boldsymbol{B}_{\mathbb{Z}}^{\mathscr{F}}$ with the family $\mathscr{F}$ of inductive nonempty subsets of $\omega$},
Algebra Discrete Math. \textbf{35} (2023). no. 1, 42--61.
DOI: 10.12958/adm2010

\bibitem{Gutik-Pozdniakova=2023b}
O. Gutik and I. Pozdniakova,
\emph{On the semigroup of endomorphisms of the monoid $\boldsymbol{B}_{\omega}^{\mathscr{F}}$ with the two-elements  family $\mathscr{F}$ of inductive nonempty subsets of $\omega$},
Visnyk Lviv. Univ. Ser. Mech.-Mat. \textbf{96} (2024), 5--24.
DOI: 10.30970/vmm.2024.96.005-024

\bibitem{Gutik-Prokhorenkova-Sekh=2021}
O. Gutik, O. Prokhorenkova, and D. Sekh,
\emph{On endomorphisms of the bicyclic semigroup and the extended bicyclic semigroup},
Visn. L'viv. Univ., Ser. Mekh.-Mat. \textbf{92} (2021), 5--16  (in Ukrainian).
DOI: 10.30970/vmm.2021.92.005-016

\bibitem{Gutik-Serivka=2023}
O. Gutik and M. Serivka,
\emph{On the semigroup of injective monoid endomorphisms of the monoid $\boldsymbol{B}_{\omega}^{\mathscr{F}^3}$ with a three element family $\mathscr{F}^3$ of inductive nonempty subsets of $\omega$},
Visn. L'viv. Univ., Ser. Mekh.-Mat. \textbf{95} (2023), 28--45. \\
DOI: 10.30970/vmm.2023.95.028-045

\bibitem{Gutik-Serivka=2025}
O. V. Gutik and M. V. Serivka,
\emph{On the semigroup of endomorphisms of the semigroup $\boldsymbol{B}_{\omega}^{\mathscr{F}^2}$ with the two-element family $\mathscr{F}^2$ of inductive nonempty subsets of $\omega$},
Nauk. Visn. Uzhgorod. Univ., Ser. Mat. Inform. \textbf{47} (2025), no. 2, 43--51 (in Ukrainian).
DOI: 10.24144/2616-7700.2025.47(2).43-51


\bibitem{Lawson=1998}
M.~Lawson,
\emph{Inverse semigroups. The theory of partial symmetries},
World Scientific, Sin\-ga\-pore, 1998.

\bibitem{Lysetska=2020}
O. Lysetska,
\emph{On feebly compact topologies on the semigroup $\boldsymbol{B}_{\omega}^{\mathscr{F}_1}$},
Visnyk Lviv. Univ. Ser. Mech.-Mat. \textbf{90} (2020), 48--56.
DOI: 10.30970/vmm.2020.90.048-056

\bibitem{Petrich-1984}
M.~Petrich,
\emph{Inverse semigroups},
John Wiley $\&$ Sons, New York, 1984.

\bibitem{Popadiuk=2022}
O. Popadiuk,
\emph{On endomorphisms of the inverse semigroup of convex order isomorphisms of the set $\omega$ of a bounded rank which are generated by Rees congruences},
Visn. L’viv. Univ., Ser. Mekh.-Mat. \textbf{93} (2022), 34--41. \\
DOI: 10.30970/vmm.2022.93.034-041

\bibitem{Wagner-1952}
V.~V. Wagner,
\textit{Generalized groups},
Dokl. Akad. Nauk SSSR \textbf{84} (1952), 1119--1122 (in Russian).

\end{thebibliography}
\end{document}